\documentclass[12pt, reqno]{amsart}
\usepackage{amsfonts,amssymb,latexsym,amsmath, amsxtra,hyperref}
\usepackage{verbatim}
\usepackage{multirow}
\hypersetup{colorlinks=true,linkcolor=blue,citecolor=magenta}

\theoremstyle{plain}
\newtheorem{theorem}{Theorem}[section]
\newtheorem*{theorem*}{Theorem}

\newtheorem{question}[theorem]{Question}
\newtheorem{lemma}[theorem]{Lemma}
\newtheorem{proposition}[theorem]{Proposition}
\newtheorem*{conjecture*}{Conjecture}

\theoremstyle{definition}

\theoremstyle{remark}

\newtheorem*{remark*}{Remark}
\newtheorem*{remarks*}{Remarks}

\numberwithin{equation}{section}

\newcommand{\Z}{\mathbb Z}
\newcommand{\C}{\mathbb C}

\def\({\left(}
\def\){\right)}

\newcommand{\im}[1]{\text{Im}#1}

\newcommand{\commenttext}[1]{}

\begin{document}

\title[Modularity of Gromov-Witten potentials]{On the modularity of certain functions from the Gromov-Witten theory of elliptic orbifolds}

\date{\today}
\author{Kathrin Bringmann, Larry Rolen, and Sander Zwegers}
\address{Mathematical Institute\\University of
Cologne\\ Weyertal 86-90 \\ 50931 Cologne \\Germany}
\email{kbringma@math.uni-koeln.de} 
\email{lrolen@math.uni-koeln.de}
\email{sander.zwegers@uni-koeln.de}
\date{\today}
\thanks{The research of the first author was supported by the Alfried Krupp Prize for Young University Teachers of the Krupp foundation and the research leading to these results has received funding from the European Research Council under the European Union's Seventh Framework Programme (FP/2007-2013) / ERC Grant agreement n. 335220 - AQSER. The second author thanks the University of Cologne and the DFG for their generous support via the University of Cologne postdoc grant DFG Grant D-72133-G-403-151001011, funded under the Institutional Strategy of the University of Cologne within the German Excellence Initiative. }

\begin{abstract}
In this paper, we study modularity of several functions which naturally arose in a recent paper of Lau and Zhou on open Gromov-Witten potentials of elliptic orbifolds. They derived a number of examples of indefinite theta functions, and we provide modular completions for several such functions which involve more complicated objects than ordinary modular forms. In particular, we give new closed formulas for special indefinite theta functions of type $(1,2)$ in terms of products of mock modular forms. This formula is also of independent interest. 
\end{abstract}
\maketitle

\section{Introduction and statement of results}
\label{IntroSection}

In the recent paper \cite{LauZhou}, Lau and Zhou studied a number of generating functions of importance in Gromov-Witten theory and mirror symmetry, and they showed modularity for several of them. To be more precise, they considered the four elliptic $\mathbb P^1$ orbifolds denoted by $\mathbb P^1_{\mathbf{a}}$ for $\mathbf{a}\in\{(3,3,3),(2,4,4),(2,3,6),(2,2,2,2)\}$. In particular, for these choices of $\bf{a}$, they explicitly computed the open Gromov-Witten potential $W_q(x,y,z)$ of $\mathbb P^1_{\bf a}$, which is in particular a polynomial in $x,y,z$ over the ring of power series in $q$ (where $q$ is interpreted as the K\"ahler parameter of the orbifold), and which is closely tied with constructions of the associated Landau-Ginzburg mirror. The reader is also referred to \cite{ChoHongKimLau,ChoHongLee} for related results, as well as to Sections 2 and 3 of \cite{LauZhou} for the definitions of the relevant geometric objects.
Lau and Zhou then proved the following in Theorem 1.1 of \cite{LauZhou}. Here as usual for $c\in \mathbb{N}$
\[
\Gamma (c) := \left\{ M= \left(\begin{smallmatrix} \alpha & \beta \\ \gamma & \delta \end{smallmatrix} \right) \in \text{SL}_2 (\mathbb{Z}); M \equiv I_2 \pmod{c} \right\} .
\]
\begin{theorem}[Lau, Zhou]\label{LauZhouTheorem}
Let $\mathbf{a}\in\{(3,3,3),(2,4,4),(2,2,2,2)\}$. Then the functions arising as the various coefficients of 
$W_q(x,y,z)$ are, up to rational powers of $q$, linear combinations of modular forms of weights $0,1/2,3/2,2$ with respect to $\Gamma(c)$ and with certain multiplier systems.
\end{theorem}
This theorem is particularly useful as it allows one to extend the potential to a certain global moduli space, and in fact this is the geometric intuition for why such a modularity statement is expected (cf. \cite{ChoHongKimLau}). Moreover, such modularity results give an efficient way to calculate complete results of the open Gromov-Witten invariants. Lau and Zhou also discussed the case of $\mathbf{a}=(2,3,6)$ and gave explicit representations for the potential $W_q(x,y,z)$. As in the discussion following Theorem 1.3 of \cite{LauZhou}, the same heuristic which shows that modularity is ``expected'' for $\mathbf{a}\in\{(3,3,3),(2,4,4)\}$ also predicts that modularity-type properties should hold and which should allow one to extend the potential to a global K\"ahler moduli space. In particular, from a geometric point of view, the case of $\mathbf{a}=(2,3,6)$ is very analogous to those cases covered in Theorem \ref{LauZhouTheorem} and it is the next simplest test case. In particular, as for $\mathbf{a}\in\{(3,3,3),(2,4,4)\}$, in this case the Seidel Lagrangian can be lifted to a number of copies of the Lagrangian for the elliptic curve of which the orbifold is a quotient. 
Motivated by these calculations and heuristics, Lau and Zhou asked the following. 
\begin{question}[Lau, Zhou]\label{LauZhouQuestion}
What are the modularity properties of the coefficients of
\\
$W_q(x,y,z)$ when $\mathbf{a}=(2,3,6)$? 
\end{question}

We describe our partial answer to Question \ref{LauZhouQuestion} in the form of several theorems which give the modular completions of several functions arising in the $(2,3,6)$ case. In each of these cases, we prove modularity by first representing the functions in terms of the $\mu$-function, the Jacobi theta function, and well-known modular forms (see Section \ref{PreliminariesSection} for the definitions).

In order to prove these results, we first establish an identity, which is also of independent interest. To state it, we let 
\[
F(z_1,z_2,z_3;\tau)
:=
q^{-\frac18}\zeta_1^{-\frac12}\zeta_2^{\frac12}\zeta_3^\frac12\left(\sum_{k>0,\,\ell,m\geq0}+\sum_{k\leq0,\,\ell,m<0}\right)(-1)^kq^{\frac{k(k+1)}2+k\ell+km+\ell m}\zeta_1^k\zeta_2^\ell\zeta_3^m
,
\]
with $q:=e^{2\pi i\tau}$ ($\tau\in\mathbb H$) and $\zeta_j:=e^{2\pi iz_j}$ ($z_j\in\C$) for $j=1,2,3$. We note that $F$ is an indefinite theta function of type $(1,2)$. 
\begin{theorem}\label{IndefThetaIdent}
For all $z_1,z_2,z_3\in \C$ with $0<\im(z_2),\im(z_3)<\im(\tau)$, we have that
\begin{equation}\label{main}
\begin{split}
F(z_1,z_2,z_3;\tau)= i\vartheta(z_1;\tau)\mu(z_1,z_2;\tau)\mu(z_1,z_3;\tau) -\frac{\eta^3(\tau) \vartheta(z_2+z_3;\tau)}{\vartheta(z_2;\tau)\vartheta(z_3;\tau)} \mu(z_1,z_2+z_3;\tau)
.
\end{split}
\end{equation}
\end{theorem}

\begin{remark*}

The right-hand side of (\ref{main}) provides a meromorphic continuation of $F$ to $\C^3$, and we frequently identify the left hand side with this meromorphic continuation implicitly. \\

\end{remark*}

Our main results can then be stated as follows, where the functions $c_{y}$, $c_{yz2}$, and $c_{yz4}$ are certain coefficients of $W_q(2,3,6)$ (see \eqref{W236Formula}).

\begin{theorem}\label{mainthm}
The function $c_y$ is modular, and $c_{yz2}$ and $c_{yz4}$ have explicit non-holomorphic modular completions $\widehat c_{yz2}$ and $\widehat c_{yz4}$. More specifically, we have:
\begin{enumerate}
\item
The function $c_y$ is a cusp form of weight $3/2$ on $\operatorname{SL}_2(\Z)$ with multiplier system $\nu_{\eta}^3$.
\item  The function $\widehat c_{yz2}$ is modular (i.e., transforms as a modular form) of weight $2$ on $\operatorname{SL}_2(\Z)$ with shadow $y^{\frac32}|\eta|^6$.
\item The function $\widehat c_{yz4}$ is modular of weight $5/2$, and is a polynomial of degree $2$ in $R(0;\tau)$ over the ring of holomorphic functions on $\mathbb H$. 
\end{enumerate}
\end{theorem}
\begin{remarks*}
\hspace*{3in}
\begin{enumerate}
\item
The explicit statements and proofs of the modularity of the functions in Theorem \ref{mainthm} are given in Section \ref{StatementProofMainThm}. \\
\item Results concerning the modularity properties of these functions could also be proven using work (in progress) of Westerholt-Raum or of Zagier and Zwegers. Moreover, the general shape of the completion of $c_{z}$ should also follow from the same works. We note that the indefinite theta function $F$ we consider here is of a degenerate type and is not representative of the generic case. Due to this degeneracy, we are able to express it in terms ``classical'' objects, which simply is not possible in the generic case.
\end{enumerate}
\end{remarks*}

The paper is organized as follows. In Section \ref{PreliminariesSection}, we collect some important facts and definitions from the theory of modular forms, Jacobi forms, and mock modular forms, and we define the functions described in Theorem \ref{mainthm}. In Section \ref{Type12IdentitySection}, we prove Theorem \ref{IndefThetaIdent}. We conclude Section \ref{StatementProofMainThm} by giving the explicit statements and proofs comprising Theorem \ref{mainthm}. 

\section*{Acknowledgements}\

The authors would like to thank Martin Westerholt-Raum and Jie Zhou for useful conversations related to the paper. The authors are also grateful to the referees for useful suggestions which improved the exposition of the paper.

\section{Preliminaries}\label{PreliminariesSection}
\subsection{Basic modular-type objects}
Throughout the paper, we require a few standard examples of modular forms and related objects. Firstly, we recall the \emph{Dedekind eta function} 
\begin{equation*}\label{EtaDefn}
\eta(\tau):=q^{\frac{1}{24}}\prod_{n\geq1}\left(1-q^n\right).
\end{equation*}
We recall that $\eta$ is a weight $1/2$ cusp form on $\operatorname{SL}_2(\Z)$ (with a multiplier which we denote by $\nu_{\eta}$). We shall also frequently use the \emph{quasimodular Eisenstein series} $E_2$, which is essentially the logarithmic derivative of $\eta$:
\begin{equation*}\label{E2Defn}
E_2(\tau):=1-24\sum_{n\geq1}\sum_{d|n}dq^{n}
.
\end{equation*}
As is well-known, $E_2$ is not a modular form, but has a slightly more complicated modularity property, known as \emph{quasimodularity}. Specifically, $E_2$ is $1$-periodic and satisfies the following near-modularity under inversion:
\begin{equation}\label{E2Trans}
\tau^{-2}E_2\left(\frac{-1}{\tau}\right)=E_2(\tau)+\frac{6}{\pi i\tau}
.
\end{equation}
Using this transformation, one can also show that the completed function 
\begin{equation}
\label{E2Hat}
\widehat E_2(\tau):=E_2(\tau)-\frac{3}{\pi v}
,
\end{equation}
where $\tau=u+iv$,
is modular of weight $2$. 
More generally we require the higher weight Eisenstein series, defined by for even natural numbers $k$ by
\[
E_k(\tau)
:=
1-\frac{2k}{B_{k}}\sum_{n\geq1}\sum_{d|n}d^{k-1}q^n
,
\]
where $B_{k}$ is the $k$-th Bernoulli number. For $k\geq4$, these 
are modular forms.

In addition to these $q$-series, we also need the \textit{Jacobi theta function}, defined by
\[
\vartheta(z;\tau)
:=
\sum_{n\in\frac12+\Z}q^{\frac{n^2}2}e^{2\pi i \left(z+\frac12\right)n}
.
\]
The \textit{Jacobi triple product identity} is the following product expansion:
\[
\vartheta(z;\tau)
=
-iq^{\frac18}\zeta^{-\frac12}
\prod_{n\geq1}
\left(1-q^n\right)
\left(1-\zeta q^{n-1}\right)
\left(1-\zeta^{-1}q^n\right)
,
\]
where $\zeta:=e^{2\pi i z}$. 
In particular, this identity implies that the zeros of $z\mapsto\vartheta(z;\tau)$, lie exactly at lattice points $z\in\Z\tau+\Z$. 
We also need the following standard formula:
\begin{equation}\label{ThetaDerivative}
\frac{\partial}{\partial z}\left[\vartheta(z;\tau)\right]_{z=0}=-2\pi\eta^3(\tau)
.
\end{equation}
Moreover, $\vartheta$ is an important example of a \emph{Jacobi form} (of weight and index $1/2$), which essentially means that it satisfies a mixture of transformation laws resembling those of elliptic functions and of modular forms. In particular, we have the following well-known transformation laws. We note that throughout we suppress $\tau$-dependencies whenever they are clear from context. 
\begin{lemma}\label{JacobiThetaTransLemma}
For $\lambda,\mu\in\Z$ and $\gamma=\left(\begin{smallmatrix}a&b\\ c&d\end{smallmatrix}\right)$, we have that
\begin{equation}\label{JacobiEllipticTrans}
\vartheta(z+\lambda\tau+\mu)=(-1)^{\lambda+\mu}q^{-\frac{\lambda^2}2}e^{-2\pi i\lambda z}\vartheta(z)
,
\end{equation}
and 
\begin{equation}\label{JacobiModularTrans}
\vartheta
\left(
\frac{z}{c\tau+d}
;
\gamma\tau
\right)
=
\nu_{\eta}^3(\gamma)(c\tau+d)^{\frac12}e^{\frac{\pi i cz^2}{c\tau+d}}\vartheta(z;\tau)
.
\end{equation}
\end{lemma}

We next prove an identity for $E_2$ in terms of an Appell-Lerch sum, which we need to compute the completion of $c_{yz4}$. We note in passing that while it is plausible that this identity has been considered before, the authors could not find a specific reference in the literature.
\begin{lemma}\label{SomeE2IdentLemma}
The following holds:
\begin{equation}\label{SomeE2Ident}
2\sum_{n\neq0}\frac{(-1)^nq^{\frac{n(n+3)}2}}{\left(1-q^n\right)^2}
=
\frac1{12}\left(E_2-1\right)
.
\end{equation}
\end{lemma}
\begin{proof}
By (7) of \cite{Zagier} and \eqref{ThetaDerivative}, we have 
\begin{equation}\label{ZagierThetaTaylorExpansion}
\frac{-2\pi\eta^3}{\vartheta(z)}=z^{-1}\operatorname{exp}\left(\sum_{n\geq1}\frac{\zeta(2n)E_{2n}z^{2n}}{n}\right)
,
\end{equation}
where $\zeta(s)$ denotes the Riemann zeta function. Using this, we directly find that the coefficient of $z^1$ in $-2\pi \eta^3/\vartheta(z)$ is $\pi^2E_2/6$.
 We now compare this with the well-known partial fraction expansion of $1/\vartheta(z)$. Namely, a standard application of the Mittag-Leffler theorem
 gives the following formula (cf. page 136 of \cite{TanneryMolk} or page 1 of \cite{RamanuajnLost}): 
 \begin{equation*}
 \begin{aligned}
\frac{\eta^3}{\vartheta(z)}
&
=
i \zeta^{\frac12}\sum_{n\in\Z}\frac{(-1)^nq^{\frac{n(n+1)}2}}{1-\zeta q^n}
.
\end{aligned}
\end{equation*}
An elementary calculation then shows that the coefficient of $z^1$ in $-2\pi\eta^3/\vartheta(z)$ is equal to 
\[
4\pi^2\left(
\sum_{n\neq0}\frac{(-1)^nq^{\frac{n(n+3)}2}}{\left(1-q^n\right)^2}
+
\frac1{24}
\right)
,
\]
which, together with the computation above, implies \eqref{SomeE2Ident}.
\end{proof}
We conclude this subsection by giving an identity for a certain quotient of theta functions in terms of an indefinite theta function which we need for the proof of Theorem \ref{mainthm} below. Throughout, we set $y_j:=\operatorname{Im}(z_j)$ for $j\in\{1,2,3\}$.

\begin{lemma}\label{ThetaQuotientIdent}
For $0<y_1,y_2<v$, we have
\begin{equation}\label{ThetaQuotientIdentFormula}
\sum_{\ell\in\Z} \frac{\zeta_1^\ell}{1-\zeta_2q^{\ell}}
=
 -i\frac{\eta^3 \vartheta(z_1+z_2)}{\vartheta(z_1)\vartheta(z_2)}
.
\end{equation}
Hence, the identity \eqref{ThetaQuotientIdentFormula} provides a meromorphic continuation of the left hand side for $z_1,z_2\in\C\setminus\left(\Z\tau+\Z\right)$.
\end{lemma}
\begin{proof}
In the given range, we may use geometric series to expand:
\begin{equation*}
\sum_{\ell\in\Z} \frac{\zeta_1^\ell}{1-\zeta_2q^{\ell}}
=
\left(\sum_{\ell,m\geq0}-\sum_{\ell,m<0}\right)\zeta_1^{\ell}\zeta_2^mq^{\ell m}
=
\sum_{\ell,m\geq1}\left(\zeta_1^{\ell}\zeta_2^m-\zeta_1^{-\ell}\zeta_2^{-m}\right)q^{\ell m}-\frac{\zeta_1\zeta_2-1}{(\zeta_1-1)(\zeta_2-1)}
.
\end{equation*}
In the notation of Theorem of Section 3 of \cite{Zagier}, this last expression is exactly $-F_{\tau}(2\pi i z_1,2\pi i z_2)$ (cf. the first line of the proof of Theorem 3 there). The result then follows directly from (vii) of Theorem 3 of \cite{Zagier} and \eqref{ThetaDerivative}.
\end{proof}

\subsection{The $\mu$ function and explicit weight $1/2$ mock modular forms}

Throughout, we require an important function used in \cite{Zwegers} to study several of Ramanujan's mock theta functions. 
The \textit{$\mu$-function} is given in terms of an Appell-Lerch series for $z_1,z_2\in\C\setminus\left(\Z\tau+\Z\right)$ and $\tau\in\mathbb H$ as 
 \[
 \mu(z_1,z_2;\tau):=\frac{\zeta_1^{\frac12}}{\vartheta(z_2)}\sum_{n\in\Z}\frac{(-\zeta_2)^nq^{\frac{n(n+1)}2}}{1-\zeta_1q^n}
 ,
 \]
where $\zeta_j:=e^{2\pi i z_j}$ ($j=1,2$).
The function $\mu$ is a \emph{mock Jacobi form}, which in particular means that it
 ``nearly'' transforms as a Jacobi form of two variables.
It turns out that $\mu$ is symmetric in $z_1$ and $z_2$ (see Proposition 1.4 of \cite{Zwegers}), i.e., that
\begin{equation}\label{MuSymmetry} 
\mu(z_1,z_2)=\mu(z_2,z_1)
,
\end{equation}
and so, for example, the ``elliptic'' transformations of $\mu$ may be summarized by the following identities.
\begin{lemma}\label{MuEllipticTransLemma}
For $z_1,z_2\in\C\setminus\left(\Z\tau+\Z\right)$, we have
\begin{equation}\label{MuTranslation1}
\mu(z_1+1,z_2)=-\mu(z_1,z_2)
,
\end{equation}
\begin{equation}\label{MuShiftTau}
\mu(z_1,z_2)
+
\zeta_2\zeta_1^{-1}q^{-\frac12}
\mu(z_1+\tau,z_2)
=
-i
\zeta_2^{\frac12}\zeta_1^{-\frac12}
q^{-\frac18}
.
\end{equation}
\end{lemma}
We note that the poles of $z_j\mapsto\mu(z_1,z_2)$ are at $z_1,z_2\in\Z\tau+\Z$, and by Lemma \ref{MuEllipticTransLemma} and \eqref{MuSymmetry} the residues are determined by 
\[
\operatorname{Res}_{z_1=0}\left(\mu(z_1,z_2)\right)
=
-\frac{1}{2\pi i \vartheta(z_2)}
.
\]

The results of \cite{Zwegers} give a completion of $\mu$ to a (non-holomorphic) Jacobi form. 
To describe this, we first require the special function $R$, given by
\[
R(z;\tau)
:=
\sum_{n\in\frac12+\Z}
\left(
\operatorname{sgn}(n)
-
E
\left(
\left(
n
+
\frac
{y}
{v}
\right)
\sqrt{2v}
\right)
\right)
(-1)^{n-\frac12}
q^{-\frac{n^2}2}
\zeta^{-n}
,
\]
where $z=x+iy$ and $E$ is the entire function 
\[
E(z)
:=
2
\int_0^z
e^{-\pi t^2}
dt
.
\]
Defining the completion 
\[
\widehat\mu(z_1,z_2)
:=
\mu(z_1,z_2)
+
\frac i2
R(z_1-z_2),
\]
Theorem 1.11 of \cite{Zwegers} shows that $\widehat\mu$ transforms like a Jacobi form. 
\begin{theorem}\label{MuCompletionTransformation}
The function $\widehat\mu$ satisfies the following:
\begin{equation*}
\begin{aligned}
\widehat \mu(z_1+k\tau+\ell,z_2+m\tau+n)
&
=
(-1)^{k+\ell+m+n}q^{\frac{(k-m)^2}2}\zeta_1^{k-m}\zeta_2^{m-k}\widehat\mu(z_1,z_2)\quad\mathrm{ for  } \ k,\ell,m,n\in\Z
,
\\
\widehat \mu
\left(
\frac{z_1}{c\tau+d}
,
\frac{z_2}{c\tau+d}
;
\frac
{a\tau+b}
{c\tau+d}
\right)
&
=
\nu_{\eta}^{-3}(\gamma)
(c\tau+d)^{\frac12}
e^
{-\frac{\pi i c(z_1-z_2)^2}{c\tau+d}}
\widehat\mu(z_1,z_2;\tau)
\ 
\mathrm{ for}
\ 
\gamma
=
\left(
\begin{smallmatrix}
a&b
\\
c&d
\end{smallmatrix}
\right)
\in\operatorname{SL}_2(\Z)
.
\end{aligned}
\end{equation*}
\end{theorem}The reason that $\mu$ is called a mock Jacobi form is closely connected to Theorem \ref{MuCompletionTransformation}. Namely, it follows directly from the theory of Jacobi forms that if $z_1$ and $z_2$ are specialized to torsion points, then the completed function $\widehat\mu$ is a \emph{harmonic Maass form} of weight $1/2$. This essentially means that in addition to transforming like a modular form of weight $1/2$, it also satisfies a nice differential equation which in particular implies that it is a real-analytic function. 
This differential equation can be phrased in terms of an important differential operator in the theory of mock modular forms. Namely, the \emph{shadow operator} $\xi_{k}:=2iv^k\overline{\frac{\partial}{\partial \overline{\tau}}}$
maps a harmonic Maass form of weight $k$ to cusp form of weight $2-k$. We are interested in computing the images of certain functions used to prove Theorem \ref{mainthm} under such operators, and for this, we require the following formula, which follows from Lemma 1.8 of \cite{Zwegers}:
\begin{equation}
\label{RShadow}
\xi_{\frac12}
\left(
R(0;\tau)
\right)
=
-\sqrt 2\eta^3(\tau)
.
\end{equation}
A \textit{mock Jacobi form} similarly is a holomorphic part of harmonic Maass-Jacobi form. It turns out that $\mu$ is essentially the holomoprhic part of a harmonic Maass-Jacobi form (see \cite{BRR}).

\subsection{Formulas of Lau and Zhou in the $(2,3,6)$ case}\label{DefinitionMainFunctionsSection}
To describe the functions occurring in Theorem \ref{mainthm}, we assume throughout that $\mathbf{a}=(2,3,6)$ and study the function $W_q(2,3,6)$ defined in \cite{LauZhou}.  Namely, noting that in the notation of \cite{LauZhou}, where $q=q_d^{48}$, and writing the resulting coefficients as functions of $\tau$, by (3.29) of \cite{LauZhou} we have 
\begin{equation}\label{W236Formula}
W_q(2,3,6)=q^{\frac18}x^2-q^{\frac1{48}}xyz+c_y(\tau)y^3+c_z(\tau)z^6+c_{yz2}(\tau)y^2z^2+c_{yz4}(\tau)yz^4,
\end{equation}
where
\begin{align*}
c_y(\tau)
:&=
q^{\frac{3}{16}}\sum_{n\geq0}(-1)^{n+1}(2n+1)q^{\frac{n(n+1)}2}
, \\
c_{yz2}(\tau)
:&=
q^{-\frac1{12}}\sum_{n\geq a\geq0}\left((-1)^{n+a}(6n-2a+8)q^{\frac{(n+2)(n+1)}{2}-\frac{a(a+1)}{2}}+(2n+4)q^{n+an+1-a^2}\right)
,\\
c_{yz4}(\tau)
:&=
q^{-\frac{17}{48}}
\sum_{
\substack
{
a,b\geq0
\\
n\geq a+b
}
}
(-1)^{n+a+b}(6n-2a-2b+7)q^{\frac{(n+1)(n+2)}2-\frac{a(a+1)}2-\frac{b(b+1)}2}
,
\end{align*}
and
$
c_{z}$
is another explicit $q$-series, which seems to be of a more complicated nature.
In Section \ref{MainThmProofSection}, we determine the modularity properties of $c_y$, $c_{yz2}$, and $c_{yz4}$. Firstly, however, we prove Theorem \ref{IndefThetaIdent}, which we need for our study of $c_{yz4}$.

\section{Statement and proof of Theorem \ref{mainthm}}\label{StatementProofMainThm}
Before stating the exact formulas and modularity properties of Theorem \ref{mainthm}, we begin with an identity of a special family of indefinite theta functions.

\subsection{A useful identity for a degenerate type $(1,2)$ indefinite theta series}\label{Type12IdentitySection}
In this section, we prove Theorem \ref{IndefThetaIdent}.

\begin{proof}[Proof of Theorem \ref{IndefThetaIdent}]

For $y_3<v$, we can use a geometric series expansion to write the left hand side of \eqref{main} as
\begin{equation*}
\begin{split}
q^{-\frac18}\zeta_1^{-\frac12}\zeta_2^{\frac12}\zeta_3^{\frac12}&\left( \sum_{\substack{ k>0,\ \ell\geq 0}}-\sum_{\substack{k\leq 0,\ \ell<0}}\right)\frac{(-1)^k q^{\frac{k(k+1)}2+k\ell}\zeta_1^k \zeta_2^\ell}{1- \zeta_3 q^{k+\ell}}\\
&=
q^{-\frac18}\zeta_1^{-\frac12}\zeta_2^{\frac12}\zeta_3^{\frac12}\sum_{k,\ell\in\Z} \rho(k-1,\ell)\, \frac{(-1)^k q^{\frac{k(k+1)}2+k\ell}\zeta_1^k \zeta_2^\ell}{1- \zeta_3 q^{k+\ell}}=:f_L(z_3)
,
\end{split}
\end{equation*}
where 
$$ \rho(k,\ell) :=\begin{cases} 1 &\text{if}\ k,\ell\geq 0,\\ -1&\text{if}\ k,\ell<0,\\ 0&\text{otherwise.}\end{cases}$$
This sum converges for all $z_3\in\C\setminus\left(\Z\tau+\Z\right)$ (as long as $y_2<v$). Now, for $0< y_2<v$, we compute that
\begin{equation*}
\begin{split}
&\zeta_2^{-\frac12}\vartheta(z_1)\mu(z_1,z_2)= \zeta_2^{-\frac12}\vartheta(z_1)\mu(z_2,z_1)=\sum_{k\in\Z} \frac{(-1)^k q^{\frac{k(k+1)}2}\zeta_1^k}{1-\zeta_2q^k}\\
&
=
\sum_{k,\ell\in\Z} \rho(k,\ell)(-1)^k q^{\frac{k(k+1)}2+k\ell}\zeta_1^k\zeta_2^\ell
=
\sum_{k,\ell\in\Z} \rho(k,\ell)(-1)^k q^{\frac{k(k+1)}2+k\ell}\zeta_1^k\zeta_2^\ell \,\frac{1-\zeta_3q^{k+\ell+1}}{1-\zeta_3q^{k+\ell+1}}\\
&
=
\sum_{k,\ell\in\Z} \rho(k,\ell)\frac{(-1)^k q^{\frac{k(k+1)}2+k\ell}\zeta_1^k\zeta_2^\ell}{1-\zeta_3q^{k+\ell+1}}
-\zeta_3\sum_{k,\ell\in\Z} \rho(k,\ell)\frac{(-1)^k q^{\frac{k(k+3)}2+1+k\ell+\ell}\zeta_1^k\zeta_2^\ell}{1-\zeta_3q^{k+\ell+1}}.
\end{split}
\end{equation*}
Using the easily checked identity
\begin{equation*}\label{RhoShiftIdent}
\rho(k,\ell)=\rho(k-1,\ell)+\delta_k
\end{equation*}
 (where $\delta_k=1$ if $k=0$ and $\delta_k=0$ otherwise) in the first sum and replacing $k$ by $k-1$ in the second, we find
\begin{equation*}
\begin{split}
\zeta_2^{-\frac12}\vartheta(z_1)\mu(z_1,z_2)&= \sum_{k,\ell\in\Z} \rho(k-1,\ell)\frac{(-1)^k q^{\frac{k(k+1)}2+k\ell}\zeta_1^k\zeta_2^\ell}{1-\zeta_3q^{k+\ell+1}}+\sum_{\ell\in\Z} \frac{\zeta_2^\ell}{1-\zeta_3q^{\ell+1}}\\
&\hspace*{5mm}+\zeta_1^{-1}\zeta_3\sum_{k,\ell\in\Z} \rho(k-1,\ell)\frac{(-1)^k q^{\frac{k(k+1)}2+k\ell}\zeta_1^k\zeta_2^\ell}{1-\zeta_3q^{k+\ell}}\\
&= q^{-\frac38}\zeta_1^{\frac12}\zeta_2^{-\frac12}\zeta_3^{-\frac12}f_L(z_3+\tau) -i\zeta_2^{-1}\frac{\eta^3 \vartheta(z_2+z_3)}{\vartheta(z_2)\vartheta(z_3)} +q^{\frac18} \zeta_1^{-\frac12}\zeta_2^{-\frac12}\zeta_3^{\frac12} f_L(z_3)
,
\end{split}
\end{equation*}
where in the second equality we used Lemma \ref{ThetaQuotientIdent}.
Some rewriting then implies that
\begin{equation}\label{L}
f_L(z_3) + q^{-\frac12} \zeta_1\zeta_3^{-1}f_L(z_3+\tau) = q^{-\frac18} \zeta_1^{\frac12} \zeta_3^{-\frac12} \vartheta(z_1) \mu(z_1,z_2) + i q^{-\frac18} \zeta_1^{\frac12} \zeta_2^{-\frac12} \zeta_3^{-\frac12}\frac{\eta^3 \vartheta(z_2+z_3)}{\vartheta(z_2)\vartheta(z_3)}.
\end{equation}

Next we consider the right hand side of \eqref{main} (as a function of $z_3$) for $z_1\not\in\Z\tau+\Z$, and define
$$ f_R(z_3) := i\vartheta(z_1)\mu(z_1,z_2)\mu(z_1,z_3) -\frac{\eta^3 \vartheta(z_2+z_3)}{\vartheta(z_2)\vartheta(z_3)} \mu(z_1,z_2+z_3).$$
Our goal is to show that $f_R$ satisfies the same transformation formula as satisfied by $f_L$ according to \eqref{L}.
This follows from a short calculation using \eqref{JacobiEllipticTrans}, \eqref{MuSymmetry}, and \eqref{MuShiftTau}, which yields
\begin{equation}\label{R}
f_R(z_3) + q^{-\frac12} \zeta_1\zeta_3^{-1}f_R(z_3+\tau) = q^{-\frac18} \zeta_1^{\frac12} \zeta_3^{-\frac12} \vartheta(z_1) \mu(z_1,z_2) + i q^{-\frac18} \zeta_1^{\frac12} \zeta_2^{-\frac12} \zeta_3^{-\frac12}\frac{\eta^3 \vartheta(z_2+z_3)}{\vartheta(z_2)\vartheta(z_3)}.
\end{equation}
Comparing \eqref{L} and \eqref{R} then gives
$$f_L(z_3)-f_R(z_3)= -q^{-\frac12} \zeta_1\zeta_3^{-1}\bigl( f_L(z_3+\tau)-f_R(z_3+\tau)\bigr)$$
and so the function $f$ given by $f(z_3):=\vartheta(z_3-z_1) \bigl( f_L(z_3)-f_R(z_3)\bigr)$ satisfies $f(z_3) =f(z_3+\tau)$. Furthermore, we also (trivially) have
$$ f_L(z_3+1) = -f_L(z_3),\qquad f_R(z_3+1) = -f_R(z_3),\qquad \text{and}\qquad f(z_3+1)=f(z_3).$$
Hence, $f$ is an elliptic function, which we aim to show is identically zero.
Both $f_L$ and $f_R$ are meromorphic functions, which could have simple poles in $\Z\tau+\Z$, but could not possibly have any other poles. In $z_3=0$, both functions actually do not have a pole: a pole of $f_L$ has to come from terms in the sum satisfying $k+\ell=0$, which does not occur for $k>0$ and $\ell\geq 0$ or for $k\leq 0$ and $\ell<0$. The functions
$ z_3\mapsto i\vartheta(z_1)\mu(z_1,z_2)\mu(z_1,z_3)$ and $z_3\mapsto\frac{\eta^3 \vartheta(z_2+z_3)}{\vartheta(z_2)\vartheta(z_3)} \mu(z_1,z_2+z_3)$
both have a simple pole in $z_3=0$ with residue $-\frac1{2\pi}\mu(z_1,z_2)$, so the residue of $f_R$ at $z_3=0$ vanishes. Hence $f$ is holomorphic in $z_3=0$ and since it is both 1- and $\tau$-periodic it is actually an entire function. By Liouville's theorem, $f$ is then constant, and since it has a zero at $z_3=z_1$, it is identically zero. 
\end{proof}

\subsection{Modularity of $c_y$}\label{MainThmProofSection}
In this section, we determine the modularity properties and explicit formulas of the functions described in Theorem \ref{mainthm}. The first function, $c_y$, is essentially a modular form, as shown in (3.42) of \cite{LauZhou}.
\begin{theorem}[Lau-Zhou]\label{ModularityCy}
The function $c_y$ is a cusp form of weight $3/2$ on $\operatorname{SL}_2(\Z)$ with multiplier system $\nu_{\eta}^3$.
\end{theorem}
\begin{remark*}
Throughout this paper, we slightly abuse terminology and refer to an object as a modular form, cusp form, etc., if it is a rational power of $q$ times such an object. 
\end{remark*}
In fact, Theorem \ref{ModularityCy} was shown \cite{LauZhou} as a consequence of the following identity.
\begin{lemma}\label{FormulaCy}
We have that
\[
c_y(\tau)
=
-q^{\frac1{16}}
\eta^3(\tau)
.
\]
\end{lemma}
\subsection{Modularity of $c_{yz2}$}
The remaining functions in Theorem \ref{mainthm} are not simply modular forms, but rather mock modular and more complicated modular-type functions.
Beginning with the $c_{yz2}$ case, and defining a natural ``corrected'' function by 
\[
\widehat c_{yz2}(\tau)
:=
q^{\frac1{12}}c_{yz2}(\tau)-\frac1{4}+\frac 32\eta^3(\tau)R(0;\tau)-\frac{1}{4}E_2(\tau)
,
\]
we show the following.
\begin{theorem}\label{ModularityCyz2}
The function $\widehat c_{yz2}$ is modular of weight $2$ on $\operatorname{SL}_2(\Z)$. In particular, $c_{yz2}$ is essentially a linear combination of products of mock modular and modular forms, and the image of $\widehat c_{yz2}$ under $\xi_{2}$
is $-\frac{3}{\sqrt2}|y^{\frac32}\eta(\tau)|^6$. 
\end{theorem}
\begin{remarks*}
\hspace*{3in}
\begin{enumerate}
\item

The ``reason'' that $\widehat c_{yz2}$ is actually modular on $\operatorname{SL}_2(\Z)$, as opposed to a congruence subgroup, is closely related to the fact that the shadow of $R(0)$ is essentially $\eta^3$, and the fact that the term $R(0)$ is therefore paired with its shadow. \\

\item
Theorem \ref{ModularityCyz2} directly gives transformation formulas for the non-completed function $c_{yz2}$, and for example can be applied to determine the asymptotic behavior of the Fourier coefficients of $c_{yz2}$. \\
\item Using the modularity of $\widehat E_2$ (defined in \eqref{E2Hat}), the last term in the definition of $\widehat c_{yz2}$ could be replaced by a multiple of $v^{-1}$. This then yields a function which is modular of weight $2$ and which has $c_{yz2}$ as its ``holomorphic'' part.
\end{enumerate}
\end{remarks*}
The modularity of $\widehat c_{yz2}$ follows immediately from Theorem \ref{MuCompletionTransformation} and from the following identity, where we denote
\[
D_z:=\frac1{2\pi i }\frac{\partial}{\partial z}
,
\quad
\quad\qquad
D_{z,0}(\cdot):=D_z(\cdot)|_{z=0}
.
\]
\begin{proposition}\label{FormulaCyz2}
We have the following:
\begin{equation*}
q^{\frac1{12}}c_{yz2}(\tau)
=
\frac12D_{z,0}\left(-\vartheta(6z;\tau)\mu(8z,6z;\tau)+\frac{\zeta^4}{1-\zeta^8}\right)+\frac1{4}\left(1- E_2(\tau)\right)
.
\end{equation*}
\end{proposition}

Deferring the proof of Proposition \ref{FormulaCyz2} to later in this section, we may now prove the modularity of $\widehat c_{yz2}$.
\begin{proof}[Proof of Theorem \ref{ModularityCyz2}]
From Proposition \ref{FormulaCyz2}, we
find directly that
\[
\widehat c_{yz2}=\frac12D_{z,0}\left(-\vartheta(6z)\widehat\mu(8z,6z)+\frac{\zeta^4}{1-\zeta^8}\right)-\frac1{3} E_2
,\]
since 
\[
D_{z,0}\left(-\frac i4\vartheta(6z)R(2z)\right)=-\frac{3i}2R(0)D_{z,0}\left(\vartheta(z)\right)=\frac32\eta^3R(0)
.
\]
Note that in the last expression, we used the fact that $\vartheta$ is an odd function of $z$.
To finish the proof, it suffices to show that 
$\widehat F_{cyz2}$
transforms like a modular form under inversion.
We note that this follows from general facts concerning differential operators acting on Jacobi forms. However, we proceed directly in this case since it is elementary. 
Namely, using \eqref{E2Trans}, Lemma \ref{JacobiThetaTransLemma}, and Theorem \ref{MuCompletionTransformation}, we compute
\begin{equation*}
\begin{aligned}
&\widehat c_{yz2}\left(\frac{-1}{\tau}\right)
=
\frac12D_{z,0}\left(-\tau e(16z^2\tau)\vartheta(6z\tau;\tau)\widehat\mu(8z\tau,6z\tau;\tau)+\frac{\zeta^4}{1-\zeta^8}\right)-\frac1{3}\tau^2E_2(\tau)-\frac{2\tau}{\pi i}
\\
&
=
\tau^2\widehat c_{yz2}(\tau)+\frac12\lim_{z\rightarrow0}\left(-32\tau^2ze(16z^2\tau)\vartheta(6z\tau;\tau)\widehat\mu(8z\tau;6z\tau)+\frac{\zeta^4}{1-\zeta^8}-\frac{\tau e^{8\pi i z\tau}}{1-e^{16\pi i z\tau}}\right)-\frac{2\tau}{\pi i}
\\
&
=
\tau^2\widehat c_{yz2}(\tau)-16\tau^2\lim_{z\rightarrow0}\left(z\vartheta(6z\tau;\tau)\mu(8z\tau;6z\tau)\right)-\frac{2\tau}{\pi i}
\\
&
=
\tau^2\widehat c_{yz2}(\tau)-16\tau^2\lim_{z\rightarrow0}\left(\frac{z}{1-e^{8\pi iz\tau}}\right)-\frac{2\tau}{\pi i}
=
\tau^2\widehat c_{yz2}(\tau)
.
\end{aligned}
\end{equation*}
In the third equality above, we used that the poles of $\widehat\mu$ only arise from $\mu$, as $R$ does not have any poles. 
The claimed formula of the image of $\widehat c_{yz2}$ under $\xi_{2}$ follows directly from \eqref{RShadow}.
\end{proof}

We now turn to the proof of Proposition \ref{FormulaCyz2}. We begin by splitting $c_{yz2}=c_{yz2,1}+c_{yz2,2}$, where
\begin{equation*}
\begin{aligned}
&
c_{yz2,1}(\tau)
:=
q^{-\frac1{12}}
\sum_{n\geq a\geq 0}
(-1)^{n+a}(6n-2a+8) q^{\frac{(n+2)(n+1)}{2}-\frac{a(a+1)}{2}}
,
\\
&
c_{yz2,2}(\tau)
:=
q^{-\frac1{12}}\sum_{n\geq a\geq0}(2n+4)q^{n+an+1-a^2}
.
\end{aligned}
\end{equation*}
We first analyze the piece $c_{yz2,1}$. 
\begin{lemma}\label{FormulaFCyz21}
We have the identity
\[q^{\frac1{12}}c_{yz2,1}(\tau)
=
D_{z,0}\left(-\vartheta(3z;\tau)\mu(4z,3z;\tau)+\frac{\zeta^2}{1-\zeta^4}\right)+\frac1{12}\left(1- E_2(\tau)\right)
.
 \]
\end{lemma}

\begin{proof}
We begin with the elementary observation that
\[
q^{\frac1{12}}c_{yz2,1}(\tau)=D_{z,0}\left(c(z;\tau)-c\left(-z;\tau\right)\right)
,
\]
where
\[
c(z;\tau):=\sum_{n\geq a\geq 0}(-1)^{n+a}\zeta^{3n-a+4}
q^{\frac{(n+2)(n+1)}{2}-\frac{a(a+1)}{2}}.
\]
Noting that
\[
\frac{(n+2)(n+1)}{2}-\frac{a(a+1)}{2}=\frac12(n+a+2)(n+1-a)
\]
and setting $j:=n+a+2$ and $\ell:=n+1-a$, we rewrite
\[
c(z)=\sum\limits_{j>\ell\geq 1\atop{\ell\equiv j+1\pmod{2}}} (-1)^j \zeta^{j+2\ell}q^{\frac{j\ell}{2}}.
\]
Splitting this sum into 2 pieces, depending on the parity of $j$, yields
\begin{equation*}
\begin{aligned}
c(z)
&
=
\sum_{j\geq \ell\geq 1}\left(\zeta^{2j+4\ell-2}q^{j(2\ell-1)}-\zeta^{2j+4\ell+1} q^{(2j+1)\ell}\right)
\\
&
=
\sum_{j\geq 1}\left(\frac{\zeta^{2j+2} q^j\left(1-\zeta^{4j} q^{2j^2}\right)}{1-\zeta^4 q^{2j}}-\frac{\zeta^{2j+5} q^{2j+1}\left(1-\zeta^{4j} q^{j(2j+1)}\right)}{1-\zeta^4 q^{2j+1}}\right)
,
\end{aligned}
\end{equation*}
where we shifted $\ell\mapsto \ell+1$ and used a geometric series expansion. 
Note that in the second summand of the last expression we can add the term $j=0$ freely as it contributes zero overall.
We now combine the second piece of each summand in the last formula as
\[
-\sum_{j\geq 1}\frac{\zeta^{6j+2} q^{2j^2+j}}{1-\zeta^4 q^{2j}}+\sum_{j\geq 0}\frac{\zeta^{6j+5} q^{2j^2+3j+1}}{1-\zeta^4 q^{2j+1}}
=\sum_{j\geq 1}\frac{(-1)^{j+1}\zeta^{3j+2} q^{\frac{j(j+1)}{2}}}{1-\zeta^4 q^j}.
\]
This term contributes the following to $c(z)-c(-z)$:
\begin{equation*}
\sum_{j\geq 1}\frac{(-1)^{j+1} \zeta^{3j+2} q^{\frac{j(j+1)}{2}}}{1-\zeta^4 q^j}-\sum_{j\geq 1}\frac{(-1)^{j+1} \zeta^{-3j-2}q^{\frac{j(j+1)}{2}}}{1-\zeta^{-4} q^j}=\sum\limits_{j\in\Z\backslash\lbrace0\rbrace}\frac{(-1)^{j+1} \zeta^{3j+2} q^{\frac{j(j+1)}{2}}}{1-\zeta^4q^j},
\end{equation*}
where we sent $j\mapsto -j$ in the second term. To consider the remaining pieces of $c$, we need to prove that
\[ D_{z,0}\left(\sum_{j\geq1}\frac{\zeta^{2j+2}q^j}{1-\zeta^4q^{2j}}-\sum_{j\geq0}\frac{\zeta^{2j+5}q^{2j+1}}{1-\zeta^4q^{2j+1}}\right)=-\sum_{n\geq1}\sigma_1(n)q^n.\] 
 To see this, we use geometric series expansions to rewrite the second piece of $c$ as
\begin{align*}
 &\sum_{\substack{j\geq1\\ \ell\geq0}}\zeta^{2j+2+4\ell}q^{j(1+2\ell)}-\sum_{j,\ell\geq0}\zeta^{2j+5+4\ell}q^{(2j+1)(\ell+1)}= \sum_{j,\ell\geq0}\left(\zeta^{2j+4+4\ell}-\zeta^{2\ell+5+4j}\right)q^{(2\ell+1)(j+1)}
\end{align*}
and let $j\mapsto j+1$ in the first sum and switch the roles of $\ell$ and $j$ in the second sum.
Differentiating with respect to $z$ and then setting $z=0$ gives
\[-\sum_{j,\ell\geq0}(2j-2\ell+1)q^{(2\ell+1)(j+1)}=-\sum_{\substack{m,n\geq1\\ m \text{ odd} }}(2n-m)q^{mn}, \]
where we set $m:=2\ell+1$ and $n:=j+1$. 
Since 
$$\sum_{\substack{m,n\geq1\\ m \text{ even} }}(2n-m)q^{mn}=2\sum_{m,n\geq1}(n-m)q^{2mn}=0,
$$ 
this is equal to 
    \[-\sum_{m,n\geq1}(2n-m)q^{mn}=-\sum_{n\geq1}\frac{nq^n}{1-q^n}=\frac1{24}\left(1-E_2\right). \]
Repeating the calculuation for the contribution at $-z$ yields the exact same expression.
Hence, we have shown that
\begin{equation}\label{Fcyz1AltFormula}
q^{\frac1{12}}c_{yz2,1}
=
 -D_{z,0}\left(\sum\limits_{j\in\Z\backslash\lbrace0\rbrace}\frac{(-1)^{j} \zeta^{3j+2} q^{\frac{j(j+1)}{2}}}{1-\zeta^4q^j}\right)+\frac{1}{12}\left(1-E_2\right)
 .
 \end{equation}
 The proof now follows directly from the definitions of $\mu$ and $\vartheta$.
\end{proof}

The following identity was proven in (3.43) of \cite{LauZhou}. This, together with Lemma \ref{FormulaFCyz21}, completes the proof of Proposition \ref{FormulaCyz2}.
\begin{lemma}\label{FormulaFCyz22}
We have the identity
\begin{equation*}\label{cyz2SecondPiece}
c_{yz2,2}(\tau)
=
\frac{q^{-\frac1{12}}}6(1-E_2(\tau))
.
\end{equation*}
\end{lemma}

\subsection{Modularity of $c_{yz4}$}
Define the corrected function 
\[
\widehat c_{yz4}(\tau)
:=
q^{\frac{11}{48}}c_{yz4}(\tau)+R(0;\tau)
\left(
-\frac{q^{\frac1{12}}c_{yz2}(\tau)}2+\frac16+\frac1{12}E_2(\tau)
\right)
+
\frac{3i}4R^2(0;\tau)\eta^3(\tau)
.
\]
Then we aim to show the following.
\begin{theorem}\label{Fcyz4CompletionModularity}
The function $\widehat c_{yz4}$ is modular of weight $5/2$. 
\end{theorem}
\begin{remark*}
It is interesting to note that there is a certain intertwining between the modularity of the different coefficients of the Gromov-Witten potential, as the function $c_{yz2}$ naturally arises in considering the completion of $c_{yz4}$, and the term $c_y$, which is essentially $\eta^3$ also occurs in the completions of $c_{yz2}$ and $c_{yz4}$. It would be interesting to see if there is a natural geometric explanation for such relations.
\end{remark*}
In order to prove this theorem, we first express the function $c_{yz4}$ in terms of the function $F$ studied in Section \ref{Type12IdentitySection}.
\begin{proposition}\label{Simplifyingcyz4Lemma}
The following identity holds:
\[
c_{yz4}(\tau)
=
-q^{-\frac{11}8}D_{z,0}
\left(
i\vartheta(3z;\tau)\mu^2(3z,2z;\tau)-\frac{\eta^3(\tau)\vartheta(4z;\tau)}{\vartheta^2(2z;\tau)}\mu(3z,4z;\tau)
\right)
.
\]
\end{proposition}
\begin{proof}
As in the proof of Proposition \ref{FormulaCyz2}, we write 
$$
q^{\frac{17}{48}}c_{yz4}(\tau) = D_{z,0}\left(f\left(z;\tau\right) - f\left(-z;\tau\right)\right)
$$
with
$$
f\left(z;\tau\right) := \sum_{\substack{a,b\geq0 \\ n\geq a+b}}(-1)^{n+a+b}\zeta^{3n-a-b+\frac72}q^{\frac{(n+2)(n+1)}{2}-\frac{a(a+1)}{2} - \frac{b(b+1)}{2}}.
$$
Setting $N:= n-a-b+1$, we compute
\[f\left(z\right)-f\left(-z\right)=\left(\sum_{\substack{a,b\geq0\\ N\geq1}}+\sum_{\substack{a,b<0\\ N\leq0}}\right)(-1)^{N+1}\zeta^{3N+2a+2b+\frac12}q^{\frac{N(N+1)}2+N(a+b)+ab}.\]
The definition of $F$ directly implies that
\[
c_{yz4}=-q^{-\frac{11}8}D_{z,0}\left(F\left(3z,2z,2z\right)\right)
.
\]
The proof then follows from Theorem \ref{IndefThetaIdent}.
\end{proof}
We are now in a position to prove the modularity of $\widehat c_{yz4}$

\begin{proof}[Proof of Theorem \ref{Fcyz4CompletionModularity}]

We first claim that
\begin{equation}\label{ClaimedCompletion}
\widehat c_{yz4}(\tau)
=
-D_{z,0}\left(i\vartheta(3z;\tau)\widehat\mu^2(3z,2z;\tau)-\frac{\eta^3(\tau)\vartheta(4z;\tau)}{\vartheta^2(2z;\tau)}\widehat\mu(3z,4z;\tau)\right)
.
\end{equation}

\noindent
For this, we use Proposition \ref{Simplifyingcyz4Lemma} to compute 
\begin{equation*}
\begin{aligned}
-
&
D_{z,0}\left(i\vartheta(3z;\tau)\widehat\mu^2(3z,2z;\tau)-\frac{\eta^3(\tau)\vartheta(4z;\tau)}{\vartheta^2(2z;\tau)}\widehat\mu(3z,4z;\tau)\right)
-
q^{\frac{11}{48}}c_{yz4}
\\
=
&
D_{z,0}
\left(
\vartheta(3z)\mu(3z,2z)R(z)
+\frac i4\vartheta(3z)R^2(z)+\frac i2\frac{\eta^3\vartheta(4z)}{\vartheta^2(2z)}R(z)
\right)
\\
=
&
D_{z,0}
\left(
\frac{\zeta R(z)}{1-\zeta^2}+R(z)\zeta\sum_{n\neq0}\frac{(-1)^nq^{\frac{n(n+1)}2}\zeta^{3n}}{1-\zeta^2q^n}+\frac i4\vartheta(3z)R^2(z)+\frac i2\frac{\eta^3\vartheta(4z)R(z)}{\vartheta^2(2z)}
\right)
.
\end{aligned}
\end{equation*}
We split this into several pieces, which we denote by 
\begin{equation*}
\begin{aligned}
H_1
&
:=
D_{z,0}
\left(
\frac{\zeta R(z)}{1-\zeta^2}
+
\frac i2\frac{\eta^3\vartheta(4z)R(z)}{\vartheta^2(2z)}
\right)
,
\\
H_2
&
:=
D_{z,0}
\left(
R(z)\sum_{n\neq0}\frac{(-1)^nq^{\frac{n(n+1)}2}\zeta^{3n+1}}{1-\zeta^2q^n}
\right)
,
\\
H_3
&
:=
\frac{i}{4} D_{z,0}
\left(
\vartheta(3z)R^2(z)
\right)
.
\end{aligned}
\end{equation*}

\noindent
Again using the fact that $R$ is even (in particular, that $R'(0)=0$), we find that 
\[
H_2
=
R(0)\left(2\sum_{n\neq0}\frac{(-1)^nq^{\frac{n(n+3)}2}}{\left(1-q^n\right)^2}
+3\sum_{n\neq0}\frac{(-1)^nnq^{\frac{n(n+1)}2}}{1-q^n}
+\sum_{n\neq0}\frac{(-1)^nq^{\frac{n(n+1)}2}}{1-q^n}
\right).
\]
The last sum in the last expression is identically zero, and by Lemma \ref{SomeE2IdentLemma}, we have
\begin{equation}\label{D2Equation}
H_2
=
\frac{R(0)}{12}\left(E_2-1\right)+3R(0)\sum_{n\neq0}\frac{(-1)^nnq^{\frac{n(n+1)}2}}{1-q^n}.
\end{equation}

\noindent 
We also directly find that 
\begin{equation}\label{D3Equation}
H_3
=
+\frac{3i}4\eta^3R^2(0).
\end{equation}

Finally, we consider the first piece $H_1$. Using \eqref{ZagierThetaTaylorExpansion} again,
we find
\begin{equation*}
\begin{aligned}
&
\frac{\zeta R(z)}{1-\zeta^2}+\frac i2\frac{\eta^3\vartheta(4z)R(z)}{\vartheta^2(2z)}
\\
&
=
\left(
-\frac{1}{4\pi iz}+\frac{\pi i }6z+O\left(z^2\right)
\right)
\left(
R(0)+\frac{R''(0)}2z^2+O\left(z^3\right)
\right)
\\
&
+
\frac1{4\pi i}
\left(
\frac1{4z^2}+2E_2\zeta(2)+O\left(z^2\right)
\right)
\left(
4z-64\zeta(2)E_2z^3+O\left(z^4\right)
\right)
\left(
R(0)+\frac{R''(0)}2z^2+O\left(z^4\right)
\right)
,
\end{aligned}
\end{equation*}
and after a short computation using \eqref{ThetaDerivative} we see that
\begin{equation}\label{D1Equation}
H_1
=
R(0)
\left(
\frac1{12}
+
\frac16E_2
\right)
.
\end{equation}

\noindent
Combining \eqref{D2Equation}, \eqref{D3Equation}, and \eqref{D1Equation} then gives
\begin{equation*}
\begin{aligned}
&
-D_{z,0}\left(i\vartheta(3z;\tau)\widehat\mu^2(3z,2z;\tau)-\frac{\eta^3(\tau)\vartheta(4z;\tau)}{\vartheta^2(2z;\tau)}\widehat\mu(3z,4z;\tau)\right)
\\
&
=
q^{\frac{11}{48}}c_{yz4}
+\frac14 R(0)E_2+3R(0)\sum_{n\neq0}\frac{(-1)^nnq^{\frac{n(n+1)}2}}{1-q^n}
+\frac{3i}4\eta^3R^2(0)
.
\end{aligned}
\end{equation*}
Using \eqref{Fcyz1AltFormula}, Lemma \ref{SomeE2IdentLemma}, and Lemma \ref{FormulaFCyz22}, we obtain
\begin{equation*}
\begin{aligned}
q^{\frac{1}{12}}c_{yz2}
&=
-\frac12D_{z,0}
\left(
\sum_{n\neq0}\frac{(-1)^n\zeta^{3n+2}q^{\frac{n(n+1)}2}}{1-\zeta^4 q^n}
\right)
+
\frac1{4}(1-E_2)
\\
&
=
-\frac12\sum_{n\neq0}\frac{(-1)^n(3n+2)q^{\frac{n(n+1)}2}}{1-q^n}-2\sum_{n\neq0}\frac{(-1)^nq^{\frac{n(n+3)}2}}{\left(1-q^n\right)^2}+\frac1{4}(1-E_2)
\\
&
=
-\frac32\sum_{n\neq0}\frac{(-1)^nnq^{\frac{n(n+1)}2}}{1-q^n}+\frac13(1-E_2)
,
\end{aligned}
\end{equation*}
which directly implies \eqref{ClaimedCompletion}. 

Hence, using \eqref{JacobiModularTrans}, \eqref{ClaimedCompletion}, and Theorem \ref{MuCompletionTransformation}, we find that
\begin{equation*}
\begin{aligned}
&\widehat c_{yz4}\left(\frac{-1}{\tau}\right)
=
D_{z,0}
\left(i\vartheta\left(3z;\frac{-1}{\tau}\right)\widehat\mu^2\left(3z,2z;\frac{-1}{\tau}\right)-\frac{\eta^3\left(\frac{-1}{\tau}\right)\vartheta\left(4z;\frac{-1}{\tau}\right)}{\vartheta^2\left(2z;\frac{-1}{\tau}\right)}\widehat\mu\left(3z,4z;\frac{-1}{\tau}\right)
\right)
\\
&
=
(-i\tau)^{\frac52}\widehat c_{yz4}(\tau)
\\
&
+
(-i\tau)^{\frac32}(28\pi i\tau)\lim_{z\rightarrow0}\left(z\left(\vartheta(3z\tau;\tau)\mu^2(3z\tau,2z\tau;\tau)+i\frac{\eta^3(\tau)\vartheta(4z\tau;\tau)}{\vartheta^2(2z\tau;\tau)}\mu(3z\tau,4z\tau;\tau)\right)\right)
,
\end{aligned}
\end{equation*}
where we used the fact that $R$ does not have a pole, so that all poles of $\widehat\mu$ come from $\mu$. The inner sum in the limit is essentially just a specialization of $F$, and by a similar computation of Laurent coefficients as in the proof of Theorem \ref{IndefThetaIdent}, it converges to a finite limit as $z\rightarrow0$. Hence, the entire limit converges to zero, and so the modularity of $\widehat c_{yz4}$ is proven.
\end{proof}

\end{document}